\documentclass[11 pt]{amsart}
\usepackage{comment}
\usepackage{enumerate}
\usepackage{amssymb, amsmath}
\usepackage{bm}

\title{Representation of integers by sparse binary forms}
\author{Shabnam Akhtari}
\address{Department of Mathematics\\
Fenton Hall\\
University of Oregon\\
Eugene, OR 97403-1222 USA}
 \email {akhtari@uoregon.edu}

 \author{Paloma Bengoechea}
\address{ ETH, Mathematics Dept.\\CH-8092, Z\"urich, Switzerland}
\email{paloma.bengoechea@math.ethz.ch}


\def\x{\langle \bm x\rangle}
\def\xx{|\bm x|}

\begin{document}

\newtheorem{thm}{Theorem}[section]
\newtheorem{prop}[thm]{Proposition}
\newtheorem{lemma}[thm]{Lemma}
\newtheorem{cor}[thm]{Corollary}
\newtheorem{conj}[thm]{Conjecture}
\newtheorem{rk}[thm]{Remark}
\newtheorem{defi}[thm]{Definition}

\begin{abstract}
We will give new upper bounds for the number of solutions to the inequalities of the shape $|F(x , y)| \leq h$, where $F(x , y)$ is a sparse binary form,  with 
integer coefficients, and $h$ is a sufficiently small integer in terms of the discriminant of the binary form $F$. Our bounds depend on the number of non-vanishing coefficients of $F(x , y)$. When $F$ is ``really sparse'', we establish a sharp upper bound for the number of solutions that is linear in terms of the number of non-vanishing coefficients. This work will provide affirmative answers to a number of conjectures posed by Mueller and Schmidt in \cite{S87} in special but important cases.
\end{abstract}
\maketitle

\section{Introduction and statements of the results}\label{Intro}

Let $F(x , y)$ be a binary form of degree $n\geq 3$ with integer  coefficients which is irreducible over the rationals.  Let $h$ be a positive integer.  By a classical result of Thue 
in  \cite{Thu}, we know that the inequality
\begin{equation}\label{Tin}
1\leq |F(x , y)| \leq h
\end{equation}
has at most finitely many solutions in integers $x$ and $y$. Such inequalities are called \emph{Thue's inequalities}. 

We will give upper bounds for the number of solutions to Thue's inequalities $|F(x , y)| \leq h$, where $F(x , y)$ is a sparse polynomial, and $h$ is sufficiently small in terms of the absolute value of the discriminant of $F$. Our bounds depend on the number of non-vanishing coefficients of the form $F$. To state  our results more precisely, let us suppose  that $F(x , y)$ is a form of degree $n
\geq 3$ which has no more than $s+1$ nonzero coefficients, so that
\begin{equation}\label{F}
F(x,y)=\sum_{i=0}^s a_i x^{n_i}y^{n-n_i}
\end{equation}
with $0=n_0<n_1<\ldots n_{s-1}<n_s=n$. We refer to such forms as sparse forms or \emph{fewnomials}.

\noindent
\textbf{Definition of Primitive Solutions}. A pair $(x , y) \in \mathbb{Z}^2$ is called a primitive solution to inequality \eqref{Tin}  if it satisfies the inequality and  
$\gcd(x , y) = 1$.  \\
We note that by this definition  $(\pm 1, 0)$ and $(0, \pm 1)$ are primitive, but, for example, if $z \neq \pm 1$, the possible solution $(z , 0)$ is not considered primitive. Suppose that 
$(x_{1}, y_{1})$ is a solution to the inequality \eqref{Tin}, then there exists an integer $t \neq 0$ such that $(x_{1}, y_{1}) = (t x_{2}, t y_{2})$ and $(x_{2}, y_{2})$ is a primitive solution of \eqref{Tin}.

Throughout this manuscript,   by $A \ll B$ we mean $A$ is bounded above by $B$ up to an explicit  constant that does not depend on any of the quantities 
$n,s,h$. Similarly, we say $A = \mathcal{O}(B)$ if $A \leq \kappa B$, for an absolute constant $\kappa$.
The following are our main theorems.

 \begin{thm}\label{newmain}
 Let $F(x,y)\in\mathbb{Z}[x,y]$ be an irreducible  binary form with  $s+1$ nonzero coefficients, degree $n> s$, and discriminant $D$. 
Let $h$ be an integer with 
\begin{equation}\label{mkappa}
0<h < \dfrac{|D|^{\frac{1}{8(n-1)}}}{(3n^{800\log^2n})^{n/2}(ns)^{2s+n}}.
\end{equation}
Let $\mathcal{N}(F, n, s, h)$ be the number of primitive solutions to the inequality 
$$1 \leq |F(x,y)|\leq h.
$$

(i) We have 
$$\mathcal{N}(F, n, s, h) \ll s\log s\min(1,\dfrac{1}{\log n -\log s}).$$

(ii)  Moreover, if $n\geq s^2$, we have 
$$\mathcal{N}(F, n, s, h) \ll s.$$
 \end{thm}

\begin{thm}\label{mainsmall}
  Let $F(x,y)\in\mathbb{Z}[x,y]$ be an irreducible  binary form with  $s+1$ nonzero coefficients, degree $n> s$, and discriminant $D$. 
Let $h$ be an integer with 
\begin{equation}\label{mkappa2} 
0 < h < \dfrac{|D| ^{\frac{1}{4(n-1)} }}{10^n \, n^{\frac{n}{4(n-1)}}}.
\end{equation}
Let $\mathcal{N}(F, n, s, h)$ be the number of primitive solutions to the inequality 
$$1 \leq |F(x,y)|\leq h.
$$
We have
$$
\mathcal{N}(F, n, s, h) \ll \sqrt{ns}.
$$
\end{thm}

The assumptions \eqref{mkappa} and \eqref{mkappa2} are quite strong, and are indeed helpful in our improvement of the previous bounds. 
 Generally one cannot expect that a  binary form of degree $n$ has such a large discriminant. However, by a result of Birch and Merriman  in \cite{BiMer}, 
 for a fixed degree $n$ only finitely many equivalence classes of irreducible binary forms of degree $n$ have bounded discriminant (see also \cite{EG91}). So our results, 
 while stated for a strong condition on the discriminant, namely \eqref{mkappa} or \eqref{mkappa2}, hold for almost all binary forms of a given degree.

In \cite{MS88} Mueller and Schmidt   obtained the upper bound
\begin{equation}\label{MSbound}
\mathcal{N}(F, n, s, h) \ll s^2h^{2/n}(1+\log h^{1/n}),
\end{equation}
for every positive integer $h$.
They could remove the logarithmic factor if $n\geq 4s$. 
One of  our contributions is to remove the 
dependency on $h$ for sufficiently small values of $h$. When $h$ is large, one naturally expects the factor $h^{2/n}$ to appear (see \cite{Mah9, ThunAnn}, for example).  After the statement of Theorem 1  in \cite{MS88}, which contains the bound \eqref{MSbound}, the authors conjecture that the factor $s^2$ should be replaced by $s$ (a conjecture originally due to Siegel). Our Theorem \ref{newmain} verifies this conjecture in case $s^2 < n$ and $h$ is small. Theorem \ref{newmain} also  improves the factor $s^2$ in \eqref{MSbound} significantly, namely by a quantity smaller than $s \log s$, again for small $h$.

Schmidt in \cite{S87} showed, for every positive integer $h$, that
\begin{equation}\label{Sbound}
\mathcal{N}(F, n, s, h) \ll (ns)^{1/2}h^{2/n}(1+\log h^{1/n}).
\end{equation}
 He conjectured that the logarithmic factor is unnecessary. 
In  \cite{Thun} Thunder replaced the factor $(1+\log h^{1/n})$ in the above bound by $(1+\log \log h /\log d)$. Using an effective result of 
Evertse and Gy\H{o}ry in \cite{EG91} on bounds for the height of binary forms in terms of their discriminant, Thunder reasons that  Schmidt's conjecture on 
unnecessariness of  $\log h^{1/n}$  holds ``essentially''. Here we will remove the dependence on $h$ in the upper bound \eqref{Sbound}, for sufficiently small $h$.

The problem of counting the number of solutions to Thue's inequalities $1 \leq \left| F(x , y) \right| \leq h$ for small 
integers $h$ has been considered previously, for example, in \cite{Sh15, Gyo1, Ste}, where upper bounds are of the shape
$c_{0} n$, where $n$ is the degree of $F$ and $c_{0}$ is an explicit constant.
 We aim for similar studies for fewnomials $F$.
Mueller in \cite{Mu} and Mueller and Schmidt in \cite{MS87} established bounds for the number of solutions of $\left| F(x , y)\right| \leq h$ for 
binomial and trinomial forms $F$.  These bounds are independent of $h$ and $n$, provided that $h$ is small in terms of $H(F)$, the maximum 
of absolute values of the coefficients of $F$. Based on their works on binomials and trinomials, Mueller and Schmidt conjectured in \cite{MS88}, 
provided that $h \leq H^{1-\frac{s}{n} - \rho}$, that the 
number of primitive solutions of \eqref{Tin}  is $\leq c(s,\rho)$, where $c(s,\rho)$  depends on $s$ and $\rho$ only.
To us it feels more natural to compare the size of $h$ with the discriminant (our methods are in sympathy with our intuition!). 
Our results in Theorems \ref{newmain} and  \ref{mainsmall} 
are under the assumption that the integer $h$ is bounded in terms of the absolute discriminant. 
 
 A very special and interesting type of fewnomials are binomials.
 In \cite{Sie2}, Siegel showed that the equation $0< |a x^n - by^n| \leq c$ has at most one primitive solution in positive integers $x$ and $y$ if
$$
\left|a b\right|^{\frac{n}{2} - 1} \geq 4 \left(n \prod_{p| n} p^{\frac{1}{p-1}}\right)^{n} c^{2n -2}.
$$
We note that the size of $c$ is compared to the discriminant, and not the height, of the binomial in Siegel's work.  Also in \cite{Eve82} Evertse extended the hypergeometric method of Siegel to give striking bounds for the number of solutions to Thue equation $a x^n - by^n = c$. These ideas have recently been generalized in \cite{AkhSS} for a larger family of Thue's inequalities which include binomial inequalities. In a breakthrough work \cite{Ben},  Bennett used a sophisticated combination of analytic methods, including the approximation tools from \cite{Sie2, Eve82},  to show 
 that
the  equation $a x^n - by^n = 1$,
with $a$ and $b$ positive,   has at most one solution in positive
integers $x, y$. This is a sharp result, as the
equation 
$(a+1)x^n-ay^n=1$
has precisely one solution $(1,1)$ in positive integers, for every positive integer $a$.

We will closely follow two fundamental works \cite{MS88, S87} of Mueller-Schmidt and Schmidt. These two papers introduce  different approximation methods 
and result in two different types of upper bounds, which are of the  shapes stated  in our Theorems \ref{newmain} and \ref{mainsmall}. 
In order to discuss how we use these ideas more clearly, we organize this article in two main parts. 
Part I explores some ideas in \cite{MS88} and \cite{Bom} further and establishes
the desired bounds in Theorem \ref{newmain}. Part II focuses on some extensions of results in \cite{S87} and contains the proof of Theorem 
\ref{mainsmall}. In both parts having the integer $h$ bounded by a function of discriminant is absolutely crucial in our application of approximation methods of \cite{MS88, S87}. Particularly in Part I, using the fact that the height of a binary form can also be bounded in terms of its absolute  discriminant, which is invariant up to $\textrm{GL}_2(\mathbb{Z})$ actions allows us to work with a given fewnomial and use some ideas from \cite{Bom}. Simply put, we control the size of all quantities that show up in classical approximation methods by the absolute value of the discriminant. 

We use two very important Lemmas established by Mueller and Schmidt, and recorded here as Propositions \ref{lmMS88} and \ref{thesetT}. Their paper \cite{MS88} includes an interesting discussion of the Newton polygon and its applications to the distribution of the roots of polynomials with only $s+1$ roots. It turns out that the roots of such polynomial are located in not more than $s$ fairly narrow annuli centered at the origin. This analysis is essential in establishing the extremely useful fact that every solution $(x , y)$ of the Thue's inequality gives a good rational approximation $\frac{x}{y}$ to a  root of the fewnomial  that belongs to a small subset of the set of all roots. The number of elements in this subset is estimated by $s$, as opposed to the general case, where one has to take into account all $n$ roots of a binary form of degree $n$.

 In our proofs in Part I (where we ultimately prove Theorem \ref{newmain}), we will need to assume that $n >  10 s$. This assumption  does not alter the statement of  Theorem \ref{newmain}, because technically if $n \leq 10 s$, the form $F(x , y)$ is not sparse and we can use a result of the first author for general Thue's inequalities in \cite{Sh15}, where the upper bound  $c_{0} n$ is established for the number of solutions of Thue's inequality $|F(x , y)| \leq h$ of degree $n$, and under the assumption \eqref{mkappa}.
 This way if $n \leq 10s$, we obtain the bound $c_{1} s $  for the number of primitive solutions to our inequality, where $c_{1} = 10\,  c_{0}$ is an explicit constant.

 This manuscript  is organized as follows.  After recalling  some basic facts and  useful theorems in Section \ref{secNot}, we will divide the article to two general parts. Part I includes Sections \ref{secpart1def}, 
 \ref{secpart1small}, \ref{secpart1med}, and  is devoted to the proof of Theorem \ref{newmain}.  Part II includes Sections \ref{secpart2def}, 
 \ref{secpart2small}, and is devoted to the proof of Theorem \ref{mainsmall}. 
 In both parts we estimate the number of solutions to our inequalities by splitting them in three or two subsets respectively: small, medium and
 large solutions for Part I,  and small and large solutions for Part II.
 The definition of small and large will differ in  Parts I and II. We  will give these definitions and other notation in Sections 
 \ref{secpart1def} 
 and \ref{secpart2def}. 
 The estimation of the number  of large solutions is not treated in details here, as some good bounds for the number of large solutions of Thue's inequalities have been 
 established in \cite{S87} and \cite{MS88}. We will define the size of solutions in a way that we can use corresponding previous results.

\section{Preliminaries}\label{secNot}

\subsection{ Discriminant, Height, and Mahler Measure}

For a binary form $G(x , y)$ that factors over $\mathbb{C}$ as
$$
\prod_{i=1}^{n} (\alpha_{i} x - \beta_{i}y),
$$
the discriminant $D(G)$ of $G$ is given by
$$
D(G) = \prod_{i<j} (\alpha_{i} \beta_{j} - \alpha_{j}\beta_{i})^2.
$$
Therefore, if we write
$$
G(x, y) = c (X - \gamma_{1}y)\ldots  (x- \gamma_{n}y),
$$ 
we have
$$
D(G) = c ^{2(n-1)} \prod_{i<j} (\gamma_{i}  - \gamma_{j})^2.
$$

The Mahler measure $M(G)$ of the form $G(x, y) = c (X - \gamma_{1}y)\ldots  (x- \gamma_{n}y)$ is defined by
$$
M(G) = |c| \prod_{i =1}^{n} \max(1 , \left| \gamma_{i}\right|).
$$
 Mahler \cite{Mah} showed
\begin{equation}\label{mahD5}
 M(G) \geq \left(\frac{|D|}{n^n}\right)^{\frac{1}{2n -2}},
 \end{equation}
where $D$ is the discriminant of $G$.
 
 Let  $G(x , y)=a_{n}x^{n} + a_{n-1}x^{n-1} y+ \ldots + a_{1}x y^{n-1} + a_{0}y^n$.
 The (naive) height of $G$, denoted by $H(G)$, is defined by 
\begin{equation}\label{height}
H(G) =\max  \left( |a_{n}|, |a_{n-1}|, \ldots , |a_{0}|\right).
\end{equation}
We have 
\begin{equation}\label{Lan5}
  {n \choose \lfloor n/2\rfloor}^{-1} H(G) \leq M(G) \leq (n+1)^{1/2} H(G).
\end{equation}
A proof of this fact can be found in \cite{Mahext}.

\subsection{ $GL_{2}(\mathbb{Z})$ Actions and Equivalent Forms}

Let 
$$
A = \left( \begin{array}{cc}
a & b \\
c & d \end{array} \right)
$$
and define the binary form $F_{A}$ by
$$
F_{A}(x , y) = F(ax + by , cx + dy).$$

We say that two binary forms $F$ and $G$ are equivalent if $G = \pm F_{A}$ for some $A \in GL_{2}(\mathbb{Z})$.

Observe that for any $2 \times 2$ matrix $A$ with integer entries
\begin{equation}\label{St6}
D(F_{A}) = (\textrm{det} A)^{n (n-1)} D(F).
\end{equation}

For  $A \in GL_{2}(\mathbb{Z})$, we have that $F_{A^{-1}} ( ax+by , cx+dy ) =  \pm F(x , y)$ and $\gcd(ax+by , cx+dy) =1 $ if and only if $\gcd(x , y) =1$. 
Therefore, the number of solutions (and the number of primitive solutions) to Thue's  inequalities does not change if we replace the binary form with an 
equivalent form.  Moreover the discriminants of two equivalent forms are equal. However,  $GL_2(\mathbb{Z})$-actions do not preserve the fact that $F$ has no more than 
$s+1$ non zero coefficients.  Also $GL_2(\mathbb{Z})$-actions do not preserve the height. So the counting problem for forms of the kind \eqref{F} does change.

To estimate the number of solutions to Thue's inequalities with fewnomials, Schmidt formulates in \cite{S87}  a condition that is invariant under
$GL_2(\mathbb{Z})$ actions.
Following Schmidt, we define a class $C(t)$ of forms of degree $n$ as follows. \\
\textbf{Definition of $C(t)$}. The set of forms
$F(x, y)$ of degree $n$ with integer coefficients, and irreducible over $\mathbb{Q}$, such that for any
reals $u,v\neq 0,0$,
the form 
\begin{equation}\label{C(t)}
uF_x+vF_y 
\end{equation}
has at most $t$ real zeros. 

Note that for $n > 0$, the irreducibility of $F$ implies that
the form \eqref{C(t)} of degree $n - 1$ is not identically zero. Also for $F\in C(t)$,
the derivative $f'(z) = F_x(z, 1)$ has fewer than $ t$ real zeros. The following is Lemma 2 of \cite{S87}.
 \begin{lemma} \label{SchL2}
 Suppose $F(x ,y)$ is irreducible of degree $n$, and  has  $s+1$ non-vanishing coefficients. Then $F(x , y) \in C(4s -2)$. 
  \end{lemma}

In Part II, we will consider inequalities of the shape $|F(x , y)| \leq h$, for forms $F \in C(4s -2)$.

\section{Part I: Strategy, outline and definitions}\label{secpart1def}

Let $\bm x = (x , y)$. We define
$$
\xx = \max (|x|,|y|),\qquad
\x = \min(|x|,|y|).$$ 
 
In order to establish our upper bounds in Theorem \ref{newmain}, we measure the size of possible solutions $(x , y)$ of our inequality by the size of 
$\x$ and $\xx$.

\textbf{Definition.} Relative to two quantities $Y_S$, $Y_L$, which will be defined below in \eqref{defY0} and \eqref{defY3}, we call a  solution $(x , y) \in \mathbb{Z}^2$ 
\begin{eqnarray*}
\textrm{small} \, \, &\textrm{if}& \, \, 0 < \x\leq Y_S, \\
\textrm{medium} \, \, &\textrm{if}& \, \,  \xx\leq Y_L \ \textrm{and}\   \x> Y_S,\\
\textrm{large} \, \, &\textrm{if}& \, \,  \xx >Y_L.
\end{eqnarray*}

We choose the constants below to be consistent with Mueller and Schmidt's work \cite{MS88}. 
Let 
\begin{equation}\label{defofR}
R=n^{800\log^2n}.
\end{equation}
From \eqref{mkappa}, we have
\begin{equation}\label{bound}
0<h\leq \dfrac{|D|^{\frac{1}{8(n-1))}}}{(3R)^{n/2}(ns)^{2s+n}}.
\end{equation}
Put
$$
C=Rh(2H\sqrt{n(n+1)})^n .
$$
Pick numbers $a,b$ with $0<a<b<1$ and let
\begin{equation}\label{deflambda}
\lambda=\dfrac{2}{(1-b)\sqrt{2/(n+a^2)}}.
\end{equation}
Note that if $a,b$ were chosen sufficiently small, then 
$$
n-\lambda=n-(\sqrt{2n+2a^2}/(1-b))>0.
$$
\textbf{Definitions of  $Y_S$ and $Y_L$.} We  define
\begin{equation}\label{defY0}
Y_S= \left( (e^6 s)^n R^{2s}  h \right)^{\frac{1}{n-2 s}},
\end{equation}
and
\begin{equation}\label{defY3}
Y_L=(2C)^{1/(n-\lambda)} (4(H(n+1)^{1/2}e^{n/2}))^{\lambda/((n-\lambda)a^2)}.
 \end{equation}
 By \eqref{bound}, we have
 $$
 (3R)^{n/2} h\leq \dfrac{|D|^{\frac{1}{8(n-1))}}}{(ns)^{2s+n}},
 $$
 which, together with \eqref{mahD5} and \eqref{Lan5}, implies
 \begin{equation}\label{Y3H}
Y_L \ll H^2.
\end{equation}

 \begin{prop}\label{psmall} 
The number of primitive small solutions of \eqref{Tin} with $F$ a fewnomial, defined  in \eqref{F}, and $h$ satisfying \eqref{bound} is no greater than $12s+16$.
\end{prop}

\begin{prop}\label{pmedium} 
Let $\mathcal{N}_{\textit{m}}(F, n, s, h)$ be the number of primitive medium solutions of \eqref{Tin} with $F$  a fewnomial, defined  in \eqref{F},  $h$ satisfying \eqref{bound}, and $n > 10s$. We have
 $$ \mathcal{N}_{\textit{m}}(F, n, s, h)\ll s \log s \min(1,\frac{1}{\log n -\log s}).$$
 Moreover, if $n\geq s^2$, we have
$$\mathcal{N}_{\textit{m}}(F, n, s, h)\ll s.$$
\end{prop}

Our assumption $n > 10 s$ in Proposition \ref{pmedium} is to make our approximation methods work more smoothly. If $n \leq 10 s$ then  a much sharper version of Theorem \ref{newmain} will be implied by the first author's previous work \cite{Sh15} on general Thue's inequalities.

\begin{prop}\label{plarge} 
Let $\mathcal{N}_{\textit{l}}(F, n, s, h)$ be the number of primitive large solutions of \eqref{Tin} with $F$ a fewnomial, defined  in \eqref{F}, and $h$ satisfying \eqref{bound}. We have
$\mathcal{N}_{\textit{l}}(F, n, s, h) \ll s$.
\end{prop}

Our definitions of  $R$, $C$, $\lambda$, $Y_S$ and  $Y_L$ are the same as in \cite{MS88}. 
  Therefore, Proposition \ref{plarge} follows directly
from \cite[Prop. 1, p.211]{MS88}. 
We prove Propositions \ref{psmall} and \ref{pmedium} in Sections \ref{secpart1small} and \ref{secpart1med}, respectively.

\section{Small solutions (I), the proof of Proposition \ref{psmall}}\label{secpart1small}

 In this section, we estimate the number of primitive solutions $\bm{x} =(x , y)$ to \eqref{Tin} with $F$ a fewnomial, defined in \eqref{F},  
 for which $0\leq \x \leq Y_{S}$ 
 with $h$ bounded by \eqref{bound}.   We will present  a number of lemmas, some variations of which have been established by others in the past. The ideas of our  proofs  can be found in Chapter III of Schmidt's book \cite{Sch-Book}.
 
We first give a bound for the number of solutions such that $0\leq y\leq Y_S$. We estimate the number of solutions with $0\leq x\leq Y_S$  similarly. 
We will regard $(x,y)$ and $(-x,-y)$ as one solution, and can assume $y \geq 0$, if we need to.

\textbf{Definition of the minimal solution}. Suppose that there is at least one  solution to \eqref{Tin} with $0\leq y\leq Y_S$ and choose
 $(x_0,y_0)$ to be a solution with minimal $y\geq 0$ (if there is no small solution then Proposition \ref{psmall} is proven).  There might be more than one  solution with minimal non-negative $y$ to the Thue's inequality. We fix one of them for our
entire argument in this section, and will denote it by $\bm{x_{0} } =(x_0,y_0)$ and call it  \emph{the minimal solution}.

\textbf{Definition of $L_i(x,y)$.} For the binary form
$$
F(x,y)=a_n(x-\alpha_1y)\ldots(x-\alpha_ny),
$$
we define
$
L_i(x,y)=x-\alpha_iy
$,
for $i=1,\dots,n$.

Let $\bm{x}_1 = (x_1 , y_1)$ and $\bm{x}_2 = (x_2 , y_2)$. We define
\begin{equation}\label{defofD}
\mathcal{D}(\bm x_1,\bm x_2)=x_1y_2-x_2y_1.
\end{equation}

\begin{lemma}\label{S56}
Suppose $\bm{x} = (x , y)\in \mathbb{Z}^2$, with $\gcd (x , y) \neq 0$ and $\bm{x}  \neq \bm{x_{0}}$, satisfies $|F(x , y)| \leq h$. We have
$$
\frac{L_i(x_0,y_0)}{L_{i}(x , y)} - \frac{L_j(x_0,y_0)}{L_{j}(x , y)} = (\beta_{j} - \beta_{i}) \mathcal{D}(\bm x,\bm x_0),
$$
where $\beta_{1}$,\ldots, $\beta_{n}$ depend on $(x,y)$ and are such that the form
$$
J(u , w) = F(x,y) (u - \beta_{1}w)\ldots (u - \beta_{n}w)
$$
is equivalent to $F$.
\end{lemma}
\begin{proof}
 This is Lemma  5 of \cite{EG16}, Lemma 
 4 of \cite{Ste} and Lemma 3 of \cite{Bom}.
\end{proof}

\begin{lemma}\label{L1}
 For each $i=1,\ldots,n$, among the  primitive solutions $(x,y)$ of \eqref{Tin} with $0<y\leq Y_S$, there is at most one such that  
$\left|L_{i}(x , y) \right| < 1/(2Y_S)$.
\end{lemma}

\begin{proof} Suppose that $(x,y)$ and $(\tilde x , \tilde y)$ are two of such distinct solutions with $y\leq \tilde y$. Then
$$
\dfrac{1}{\tilde y y}
\leq \left|\dfrac{\tilde x}{\tilde y}-\dfrac{x}{y}\right|
\leq \left|\dfrac{\tilde x}{\tilde y}-\alpha_i\right|+\left|\dfrac{x}{y}-\alpha_i\right|
< \dfrac{1}{yY_S},
$$
 so that $\tilde y > Y_S$, which contradicts the assumption above.
 \end{proof}

 Suppose, without loss of generality, that
 \begin{equation}\label{fixL1}
 \left|L_1(x_0,y_0)\right|=\min_{1\leq i\leq n}(\left|L_i(x_0,y_0)\right|).
 \end{equation}
  Note that, since $|F(x_0,y_0)|\leq h$, 
 \begin{equation}\label{1}
 \left|L_1(x_0,y_0)\right|\leq h^{1/n}.
 \end{equation}
 
By Lemma \ref{L1}, there might exist a unique primitive solution $(x^*, y^*)$ such that $ 0 < y^* \leq Y_{S}$ and 
$$
\left|L_{1}(x^* , y^*) \right| < \dfrac{1}{2Y_S}.
$$
We define the set
 \begin{equation}\label{defofsetA}
 \bm{A} = \{ (x_{0} , y_{0}) , (x^*, y^*)\}.
 \end{equation}
 We note that $ 1\leq  \left|\bm{A} \right| \leq 2$.

By Lemma \ref{L1}, for any solution   $(x , y)\neq (x^*,y^*)$   with $0<y\leq Y_S$, we have
\begin{equation}\label{j}
\left|L_{1}(x , y) \right| \geq \dfrac{1}{2Y_S}.
\end{equation}
By Lemma \ref{S56} and \eqref{j},
\begin{equation}\label{S57}
\frac{|L_i(x_0,y_0)|}{\left|L_{i}(x , y) \right|} \geq |\beta_{1} - \beta_{i}| |\mathcal{D}(\bm x,\bm x_0)| - 2Y_S |L_1(x_0,y_0)|,
\end{equation}
where $\mathcal{D}$ is defined in \eqref{defofD}.
For the complex conjugate  $\bar{\beta_{1}}$  of $\beta_{1}$, we also have
$$  
\frac{|L_i(x_0,y_0)|}{\left|L_{i}(x , y) \right|} \geq |\bar{\beta_{1}} - \beta_{i}| |\mathcal{D}(\bm x,\bm x_0)| - 2Y_S |L_1(x_0,y_0)|.
$$
Hence
$$
\frac{|L_i(x_0,y_0)|}{\left|L_{i}(x , y) \right|} \geq |\textrm{Re}(\beta_{1}) - \beta_{i}| |\mathcal{D}(\bm x,\bm x_0)| - 2Y_S |L_1(x_0,y_0)|,
$$
where $\textrm{Re}(\beta_{1})$ is the real part of $\beta_{1}$.
Now we choose an integer $m = m(x , y)$, with $|\textrm{Re}(\beta_{1}) -m|\leq 1/2$, and we obtain 
\begin{equation}\label{S58}
\frac{|L_i(x_0,y_0)|}{\left|L_{i}(x , y) \right|} \geq \left(|m- \beta_{i}| -\frac{1}{2}\right) |\mathcal{D}(\bm x,\bm x_0)| -2Y_S |L_1(x_0,y_0)|,
\end{equation}
for $i = 1,\ldots , n$.
\\

\noindent \textbf{Definition of the sets  $\frak{X}_{i}$}.
Let $\frak{X}_{i}$ be the set of solutions $(x,y)\not\in\bm A$ with $1\leq y\leq Y_S$ and $|L_i(x,y)|\leq \frac{1}{2y}$, where $1\leq i \leq n$.

We note that if $\alpha_{i}$ and $\alpha_{j}$ are complex conjugates then $\frak{X}_{i} = \frak{X}_{j}$.

\begin{lemma}\label{Sl5}
 Suppose $(x_{1} , y_{1})$ and $( x_{2} , y_{2})$ are two distinct solutions in $\frak{X}_{i}$,  with $0< y_{1} \leq  y_{2}$ . Then
$$
\frac{ y_{2}}{y_{1}} \geq \frac{1}{2Y_S+7/3} \max(1 , |\beta_{i}(x_1 , y_1) - m(x_1 , y_1)|).
$$
\end{lemma}
\begin{proof} We follow the proof of Lemma 4 of \cite{Bom}. 
We have that 
\begin{align}\label{DxSl5}
1\leq 
\left| y_2 x_1 - y_1   x_2\right|
&\leq y_1 |L_i( x_2, y_2)|+ y_2 \left| L_i(x_1 ,y_1)\right|\\
&\leq\dfrac{y_1}{2 y_2}+  y_2 \left|L_i(x_1,y_1)\right| \nonumber \\
&\leq \dfrac{1}{2}+ y_2 \left|L_i(x_1,y_1)\right|. \nonumber
\end{align}
Therefore,
 $$ 
 y_2\geq \frac{1}{2 \left| L_i(x_1,y_1)\right|}.
 $$
  Combining this with \eqref{S58}, we get
\begin{equation}\label{pre32}
\frac{ y_2}{y_1} \geq \dfrac{1}{2}\Big(|m-\beta_i|-\dfrac{1}{2}\Big)\dfrac{|\mathcal{D}(\bm x_1,\bm x_0)|}{y_1|L_i(x_0,y_0)|} - \dfrac{Y_S|L_1(x_0,y_0)|}{y_1|L_i(x_0,y_0)|},
\end{equation}
where $\beta_{i} = \beta_{i}(x_{1}, y_{1})$ are introduced in Lemma \ref{S56}, $m = m(x_1 , y_1)$ is an integer satisfying  $|\textrm{Re}(\beta_{1}) -m|\leq 1/2$, with $\beta_1 = \beta_{1}(x_{1}, y_{1})$.

Now, by \eqref{fixL1}, we have that $|L_1(x_0,y_0)|\leq |L_i(x_0,y_0)|$
and
\begin{eqnarray}\label{32}\nonumber
\dfrac{|\mathcal{D}(\bm x_1,\bm x_0)|}{y_1 |L_i(x_0,y_0)|} &=&\dfrac{\left|\dfrac{x_1}{y_1}-\dfrac{x_0}{y_0}\right|}{\left|\dfrac{x_0}{y_0}-\alpha_i\right|}
\geq
\dfrac{\left|\dfrac{x_1}{y_1}-\dfrac{x_0}{y_0}\right|}{\left|\dfrac{x_0}{y_0}-\dfrac{x_1}{y_1}\right|+\dfrac{1}{2y_{1}^2}}\\
&\geq&
\dfrac{1}{1+\dfrac{1}{2y_{1}^2\left|\dfrac{x_0}{y_0}-\dfrac{x_1}{y_1}\right|}}\geq \dfrac{2}{3},
\end{eqnarray}
where the last inequality is because  $\left|\dfrac{x_0}{y_0}-\dfrac{x_1}{y_1}\right|\geq \dfrac{1}{y_0y_1}$ and 
$$
\dfrac{1}{2y_{1}^2\left|\dfrac{x_0}{y_0}-\dfrac{x_1}{y_1}\right|}\leq \dfrac{y_0}{2y_1}\leq \dfrac{1}{2}.
$$
Therefore, by \eqref{pre32} and \eqref{32}, we have
\begin{align*}
\frac{ y_2}{y_1} &\geq \max\Big(1,\Big(|m-\beta_i|-\dfrac{1}{2}\Big)\dfrac{1}{3} - Y_S\Big)
\geq \max\Big(1,\dfrac{2}{3}(|m-\beta_i|\Big)\dfrac{1}{\frac{7}{3}+2Y_S},
\end{align*}
where in the second inequality  we used that $\max(1,\frac{\zeta}{2}-a)\geq \frac{1}{2a+2} \max(1,\zeta)$ with $\zeta=\frac{2}{3}|m-\beta_i|$ and 
$a=\frac{1}{6} + Y_S$.
\end{proof}

\begin{lemma}\label{Sl6}
Suppose $(x , y) \in \mathbb{Z}^2$, and $ (x , y) \not \in \bm A$,  with $0< y \leq Y_{S}$ and $\gcd(x , y) =1$,  satisfies $|F(x , y)| \leq h$  and 
$\left|L_{i}(x , y) \right|> \frac{1}{2y}$. Then 
$$
|m(x , y) - \beta_{i}(x , y)| \leq \dfrac{7}{2} + 2 \, h^{1/n} \, Y_S.
$$
\end{lemma}
\begin{proof}
 By \eqref{S58},  we have 
$$
|m-\beta_i|\leq \Big(\dfrac{|L_i(x_0,y_0)|}{|L_i(x,y)|} + 2Y_S |L_1(x_0,y_0)|\Big)\dfrac{1}{|\mathcal{D}(\bm x,\bm x_0)|}+\dfrac{1}{2}.
$$
Since $\left|\dfrac{x}{y}-\dfrac{x_0}{y_0}\right|\geq\dfrac{1}{yy_0}$ and we are assuming $|L_i(x,y)|>\frac{1}{2y}$, we have
$$
\dfrac{|L_i(x_0,y_0)|}{|L_i(x,y)| |\mathcal{D}(\bm x,\bm x_0)|}
\leq \dfrac{\left|\alpha_i-\dfrac{x}{y}\right| + \left|\dfrac{x}{y}-\dfrac{x_0}{y_0}\right|}{y^2 \left|\alpha_i-\dfrac{x}{y}\right| \left|\dfrac{x}{y}-\dfrac{x_0}{y_0}\right|}
\leq \dfrac{1}{y^2 \left|\dfrac{x}{y}-\dfrac{x_0}{y_0}\right|} + \dfrac{1}{y^2 \left|\alpha_i-\dfrac{x}{y}\right|}
\leq 3.
$$
Therefore, using also that  $|L_1(x_0,y_0)|\leq h^{1/n}$, by \eqref{1},  and since $|\mathcal{D}(\bm x,\bm x_0)|\geq 1$, we conclude that
$$
|m-\beta_i|\leq \dfrac{7}{2} + 2 \, h^{1/n} Y_S .
$$
\end{proof}

Let  $(x_1,y_1)\in \mathbb{Z}^2$ be a fixed solution to the inequality $|F(x , y)| \leq h$. The form 
$$
J(u , w) = F(x_1 , y_1) (u - \beta_{1}w)\ldots (u - \beta_{n}w)
$$
is equivalent to $F$ by Lemma \ref{S56}. Therefore the form 
$$
\hat{J}(u , w) =  F(x_1 , y_1) (u - (\beta_{1}-m) w)\ldots (u - (\beta_{n}- m)w),
$$
which is the translation of $J$ by $m=m(x_1,y_1) \in \mathbb{Z}$, 
is also equivalent to $F$. Hence by \eqref{mahD5}, 
\begin{equation}\label{Spre60}
\prod_{i=1}^{n} \max(1, |\beta_{i}(x _1, y_1)-m(x_1 , y_1)|) \geq \frac{M(\hat{J})}{|F(x_1,y_1)|} \geq \dfrac{|D|^{\frac{1}{2n-2}}}{hn^n}.
\end{equation}

\textbf{Definition of $\frak{X}$}. For each set $\frak{X}_{i}$, ($i = 1, \ldots, n$)  that is not empty,  let  $(x^{(i)} , y^{(i)}) \in \frak{X}_{i}$ be the element with the largest value of $y$.
Let $\frak{X}$ be the set of solutions of $|F(x , y)| \leq h$ that are not in $\bm A$ and with $1 \leq y \leq Y_{S}$ except the elements 
$(x^{(1)} , y^{(1)})$, \ldots, $(x^{(n)} , y^{(n)})$.

In order to estimate the number of elements in $\frak{X}$, and in view of \eqref{Spre60}, we will give an upper bound for the product 
$$\prod_{i=1}^{n} \max(1, |\beta_{i}(x _1, y_1)-m(x_1 , y_1)|),
$$ 
for every $(x_{1}, y_{1}) \in \frak{X}$. 
\begin{lemma}
For any fixed $i \in \{1,\ldots, n\}$, we have
\begin{eqnarray}\label{E3}
 \prod_{(x , y) \in  \mathfrak{X}} \frac{\max\left(1, \left|\beta_{i}(x, y) - m(x, y) \right|\right)}{Y_S (2Y_S+7/2)}  
 \leq 1.
\end{eqnarray}
\end{lemma}
\begin{proof}
Suppose that
 the set   $\frak{X}_{i}$ is non-empty. We index the elements of $\frak{X}_{i}$ as 
$$
(x_{1}^{(i)}, y_{1}^{(i)}), \ldots, (x_{v}^{(i)}, y_{v}^{(i)}),
$$
 so that $y_{1}^{(i)} \leq \ldots \leq y_{v}^{(i)}$ (note that $(x_{v}^{(i)}, y_{v}^{(i)}) = (x^{(i)} , y^{(i)})$). By Lemma \ref{Sl5},
\begin{equation}\label{yk}
\frac{1}{2Y_S + 7/2} \max\left(1, \left|\beta_{i}(x_{k}^{(i)}, y_{k}^{(i)}) - m(x_{k}^{(i)}, y_{k}^{(i)})\right|\right) \leq \frac{y_{k+1}^{(i)}}{y_{k}^{(i)}}
\end{equation}
for $k = 1 \ldots, v-1$. 
Therefore,
\begin{equation}\label{E1touse}
\prod_{(x , y) \in  \frak{X} \bigcap \frak{X_{i}} }\frac{1}{2Y_S+7/2}  \max\left(1, \left|\beta_{i}(x, y) - m(x, y) \right|\right) \leq Y_S. 
\end{equation}

For any solution $(x , y)\in \frak{X}$, with $1\leq y\leq Y_S$,  that does not belong to  $\frak{X_{i}}$,  by Lemma \ref{Sl6}, we have
\begin{equation}\label{E2}
\frac{ \max\left(1, \left|\beta_{i}(x, y) - m(x, y) \right|\right)}{2\, h^{1/n} Y_S+7/2}   \leq 1.
\end{equation}
This, together with \eqref{E1touse}, completes the proof of Lemma.
\end{proof}

Next we will use the following striking result from \cite{MS88} to establish inequalities similar to \eqref{E3} for the solutions $(x^{(i)}, y^{(i)})$ which, by definition, do not belong to $\frak{X}$.
\begin{prop}\label{lmMS88}
Let $F$ be a fewnomial, defined in \eqref{F}. There is a set $\bm{S}$ of roots $\alpha_i$ of  $F(x,1)$ with  $\left|\bm{S}\right|\leq 6s+4$ such that 
for any real  $\zeta$,
$$
\min_{\alpha_\ell\in S}|\zeta-\alpha_\ell| \leq R \min_{1\leq i\leq n}|\zeta-\alpha_i|.
$$
\end{prop}
\begin{proof} This is Lemma 7 of \cite{MS88}.
\end{proof}

Let
$$
\bm{S_{1}} = \bm{S} \cup \{\alpha_{1}\},
$$
where $\bm{S}$ is the set in  Proposition \ref{lmMS88}, and $\alpha_{1}$ is the fixed root associated to the minimal solution, for which  the inequality \eqref{fixL1} is satisfied. Proposition \ref{lmMS88} implies that 
$$
\left|\bm{S_{1}}\right| \leq 6s +5.
$$
Let $\bm{S_{1}} = \{\alpha_{1}, \alpha_{2}, \ldots, \alpha_{t}\}$, with $1\leq t \leq 6s +5$.

Recall that 
we denote by $(x^{(i)},y^{(i)})$ the element in $\mathfrak{X}_i$ with the largest value of $y$.
Suppose $(x^{(i)},y^{(i)})\in \mathfrak{X}_i\backslash\left\{\mathfrak{X}_1\cup\ldots\cup\mathfrak{X}_{i-1}\right\}$, for  $ t < i \leq n$.

By Proposition \ref{lmMS88}, there exists $\ell\in\left\{1,\ldots,t\right\}$ such that
$$
\left|L_i(x^{(i)},y^{(i)})\right|\geq \dfrac{\left|L_\ell(x^{(i)},y^{(i)})\right|}{R} 
\geq \dfrac{1}{2y^{(i)}R},
$$
where the last inequality is because $(x^{(i)},y^{(i)}) \not \in \mathfrak{X}_{\ell}$.  Combining this with \eqref{S58}, we obtain
$$
\left|m(x^{(i)},y^{(i)})-\beta_{i}(x^{(i)},y^{(i)})\right| \leq\dfrac{2|L_i(x_0,y_0)|y^{(i)}R + 2Y_S|L_1(x_0,y_0)|}{|\mathcal{D}(\bm x^{(i)},\bm x_0)|}+\dfrac{1}{2}.
$$
Using \eqref{32} and $|L_1(x_0,y_0)|\leq h^{1/n}$, we obtain
$$
\left|m(x^{(i)},y^{(i)})-\beta_{i}(x^{(i)},y^{(i)})\right| \leq 3R+2 h^{1/n} \, Y_S+\dfrac{1}{2}.
$$
This, together with \eqref{bound} and \eqref{defY0}, implies that
\begin{eqnarray}\label{extra}
& & \frac{\left|m(x^{(i)},y^{(i)})-\beta_{i}(x^{(i)},y^{(i)})\right|}{Y_S(2Y_S+7/2)}\\ \nonumber
&\leq& \frac{ 3R+2 h^{1/n} \, Y_S+\dfrac{1}{2} }{Y_S(2Y_S+7/2)}\\ \nonumber
& < & |D|^{\frac{1}{4 n(n-1)}}.
\end{eqnarray}

\textbf{Definition of the set $\frak{X}^+$}.
Let 
 $ 
 \frak{X}^+ = \frak{X}\cup\left\{(x^{(i)},y^{(i)})\right\}_{t<i\leq n}
 $.

By \eqref{Spre60}, \eqref{E3}, \eqref{extra} and Lemma \ref{Sl6}, we have that
\begin{eqnarray}\label{theineqR}
\, \,\, \, \, \, & & \, \, \, \, \, \, \, \, \,  \, \, \, \, \, \, \, \, \, \, \, \, \,  \, \, \, \, \, \, \, \left(  \frac{|D|^{\frac{1}{2n-2}}}{h\left[nY_S (2Y_S+7/2)\right]^n}\right)^{ |\frak{X}^+|} \\ \nonumber
&\leq & \prod_{(x , y) \in\frak{X}^+}\frac{1}{\left[Y_S (2Y_S+7/2)\right]^n}  \prod_{i=1}^{n} \max(1, |\beta_{i}(x _1, y_1)-m(x_1 , y_1)|)\\ \nonumber
&< & \left( |D|^{\frac{1}{4 n(n-1)}}\right)^{n-t}.
\end{eqnarray}
By \eqref{bound} and \eqref{defY0}, we have
\begin{eqnarray*}
h\left[nY_S (2Y_S+7/2)\right]^n &<& h (2n)^n \left( {Y_S}^{2}+ 2 Y_S\right)^n\\ \nonumber
& < & (2n)^n h \left(Y_S +1 \right)^{2n}  =  (2n)^n h \left(\left( (e^6 s)^n R^{2s}  h \right)^{\frac{1}{n-2 s}} + 1\right)^{2n} \\ \nonumber
& < &(2n)^n h^{1+ \frac{2n}{n -2s}} \left(\left( (e^6 s)^n R^{2s}  \right)^{\frac{1}{n-2 s}} + 1\right)^{2n} \\ \nonumber
& < & (2n)^n  \left(\dfrac{|D|^{\frac{1}{8(n-1))}}}{(3R)^{n/2}(ns)^{2s+n}}\right)^{1+ \frac{2n}{n -2s}} \left(\left( (e^6 s)^n R^{2s}  \right)^{\frac{1}{n-2 s}} + 1\right)^{2n}.
\end{eqnarray*}
Since  we assumed $10s < n$, we have the following inequalities for the exponents of $|D|$, $R$, $s$ in the last expression above.
$$
\frac{1}{8 n(n-1)} \left(1+ \frac{2n}{n -2s}\right) < \frac{1}{8 n(n-1)} \left(\frac{7}{2}\right), 
$$
 $$ \frac{4ns}{n-2s} <   \frac{n}{2} \left(1+ \frac{2n}{n -2s}\right)
 $$
 and
 $$
 \frac{2 n^2}{n-2s} < (2s+n)  \left(1+ \frac{2n}{n -2s}\right).
 $$
 Therefore, 
\begin{equation}\label{D716}
h\left[nY_S (2Y_S+7/2)\right]^n <  |D|^{\frac{7}{16(n-1)}}.
\end{equation}
By \eqref{theineqR} and \eqref{D716}, we have
$$
\left(  \frac{|D|^{\frac{1}{2n-2}}}{|D|^{\frac{7}{16(n-1)}}}\right)^{ |\frak{X}^+|}  <  \left( |D|^{\frac{1}{8 n(n-1)}}\right)^{n}.
$$
Therefore,
$$
|\frak{X}^+| < \frac{\frac{1}{4 (n-1)} \log|D|}{\frac{1}{16 (n-1)} \log|D|} = 4.
$$

The solutions $(x,y)$ with $0<y\leq Y_S$ are either in  $\frak{X}^+ \cup \bm{A} $ (see \eqref{defofsetA} for definition) or they are among possible $(x^{(1)},y^{(1)}) ,\ldots, (x^{(t)}, y^{(t)})$. Counting the solutions 
in the set $\bm{A}$, and by Proposition  \ref{lmMS88}, we see that the total number of solutions with $0\leq y\leq Y_S$ is  no greater than $6s+6+ 3+3$.
The number of solutions of \eqref{Tin} such that $0\leq x\leq Y_S$ can be estimated in a similar way, by considering the form 
$$
F(x,y)=a_0(y-\gamma_1 x)\cdot\ldots\cdot(y-\gamma_n y)
$$
and putting $L_i(x,y)=y-\gamma_i x$. Here $\gamma_1,\ldots,\gamma_n$ are the roots of the polynomial $F(1,y)$. 
We conclude that the number of solutions with $0\leq x\leq Y_S$ is no greater than
$6s+12$.

\section{Medium solutions (I), proof of proposition \ref{pmedium} }\label{secpart1med}

We divide the interval $ [Y_{S}, Y_{L}]$ into $N+1$ subintervals, where $Y_{S}$ and $Y_{L}$ are defined in \eqref{defY0} and \eqref{defY3} and $N$ depends on $s$ and is defined below. We will show there are only few solutions $(x , y)$ with $y$   in each of these subintervals. In this section we will assume $n > 10 s$.
We define  a positive integer $N = N(n, s)$ as follows.\\
If $n \geq s^2$, we put $N=2$.\\
If $n< s^2$, we choose $N \in \mathbb{N}$ such that
\begin{equation}\label{N}
 10 s^{1+\frac{1}{N}}\leq n\leq 10 s^{1+\frac{1}{N-1}}.
\end{equation}
 If $n<s^2$, we have
\begin{equation}\label{bN}
N\leq \dfrac{\log s}{\log n - \log s}.
\end{equation}

For $\ell=1,\ldots,N$, we define 
$$
Y_\ell=Y_S H^{\frac{1}{s^{1-(\ell-1)/N}}},
$$
where $H$ the height of $F$, defined in \eqref{height}.
We put 
$$Y_0=Y_S\, \qquad  \textrm{and} \, \qquad Y_{N+1}=Y_L.
$$

We will use the following important result achieved in \cite{MS88}.
\begin{prop}\label{thesetT}
 There is a set $\bm T$ of roots of $F(x,1)$ and a set $\bm T^\ast$ of roots of $F(1,y)$, both with cardinalities at most $6s+4$, 
such that any solution $(x,y)$ of \eqref{Tin} with $\x\geq Y_S$ either has 
\begin{equation}\label{17.1}
\left|\alpha-\dfrac{x}{y}\right|<\dfrac{R(ns)^2}{H^{(1/s)-(1/n)}}\left(\dfrac{(4e^3s)^n h}{y^{n}}\right)^{1/s}
\end{equation}
with some $\alpha\in \bm T$, or has
\begin{equation}\label{17.2}
\left|\alpha^\ast-\dfrac{y}{x}\right|<\dfrac{R(ns)^2}{H^{(1/s)-(1/n)}}\left(\dfrac{(4e^3s)^n h}{x^{n}}\right)^{1/s}
\end{equation}
for some $\alpha^\ast\in \bm T^\ast$.
\end{prop}
\begin{proof}
This is  Lemma 17 of \cite{MS88}.
\end{proof}

Let $\alpha\in \bm T$. 
For $\ell\in\left\{0,\ldots,N\right\}$, let
$(x_1, y_1),\ldots, (x_{w_{\ell}}, y_{w_{\ell}})$ be the  primitive solutions of our inequality, with  $Y_{\ell} < y_i\leq Y_{\ell+1}$, satisfying  \eqref{17.1} and ordered so that
$$
 Y_{\ell}< y_{1}\leq
\ldots\leq y_{w_{\ell}} \leq Y_{\ell+1}.
$$
By \eqref{17.1},
we have that
$$
\dfrac{1}{y_i y_{i+1}} \leq \left|\dfrac{x_{i+1}}{y_{i+1}}-\dfrac{x_i}{y_i}\right|\leq \dfrac{K}{ H^{(1/s)-(1/n)} y_i^{\frac{n}{s}}},
$$
with 
$$
K=2R(ns)^2 (4e^3s)^{n/s} h^{1/s}.
$$
Therefore, for solutions $(x, y)$ with $y \in [Y_{\ell}, Y_{\ell+1}]$, we have
\begin{equation}\label{yi1}
y_{i+1}\geq K^ {-1} H^{\frac{1}{s}-\frac{1}{n}} y_i^{\frac{n}{s}-1}\geq K^{-1} H^{\frac{1}{s}-\frac{1}{n}} Y_\ell^{\frac{n}{s}-2}y_i.
\end{equation}

First we will give an estimate for the number of primitive solutions in the first subinterval $[Y_{0}, Y_{1}]$. We have $Y_{0} = Y_{S}$. By the definition \eqref{defY0} of $Y_S$,
we have
\begin{equation}\label{yi2}
 K^{-1}Y_S^{\frac{n}{s}-2} \geq 1.
\end{equation}

For $\ell=0$, we have by \eqref{yi1} and \eqref{yi2} that $y_{i+1}\geq H^{\frac{1}{s}-\frac{1}{n}} y_i$, so 
$y_{w_0}\geq (H^{\frac{1}{s}-\frac{1}{n}})^{(w_0-1)} y_1$.  Therefore, we have
$$
Y_{1} \geq y_{w_0}\geq (H^{\frac{1}{s}-\frac{1}{n}})^{(w_0-1)} Y_{0},
$$
and
$$
w_0-1\leq \dfrac{\log \frac{Y_1}{Y_0}}{(\frac{1}{s}-\frac{1}{n})\log H} \leq \frac{10}{9},
$$
since 
$\log \frac{Y_1}{Y_0} = \frac{1}{s}\log H$ and $n> 10 s$.

For $1\leq \ell < N$, by \eqref{yi1} and \eqref{yi2} we have that
$$
 y_{i+1}\geq K^{-1} H^{\frac{1}{s}-\frac{1}{n}} Y_S^{\frac{n}{s}-2} H^{\frac{n/s-2}{s^{1-(\ell-1)/N}}} y_i
 \geq H^{\frac{n}{s^{2-(\ell-1)/N}}-\frac{2}{s^{1-(\ell-1)/N}}} y_i.
 $$
Therefore,
$$
y_{w_{\ell}}\geq  H^{\left(\frac{n}{s^{2-(\ell-1)/N}}-\frac{2}{s^{1-(\ell-1)/N}}\right)(w_{\ell}-1)}y_1,
$$
and since $Y_{\ell} < y_{1} < y_{w_{\ell}} \leq Y_{\ell+1}$, we have
$$
w_{\ell}-1\leq \dfrac{\log \frac{Y_{\ell+1}}{Y_\ell}}{(\frac{n}{s^{2-(\ell-1)/N}}-\frac{2}{s^{1-(\ell-1)/N}})\log H}.
$$
For $\ell < N$, since $\log \frac{Y_{\ell+1}}{Y_\ell} < \frac{1}{s^{1-\ell/N}}\log H$ and $n\geq 10 s^{1+1/N}$,
$$
w_{\ell}-1\leq
\dfrac{1}{\frac{n}{s^{1+1/N}}-\frac{2}{s^{1/N}}} < 1.
$$
For $\ell=N$, we have  $\log Y_{\ell+1}=\log Y_L  \ll 2 \log H$, by \eqref{Y3H}, and therefore we get 
$$
w_{N}-1\ll \dfrac{2}{\frac{n}{s^{1+1/N}}-\frac{2}{s^{1/N}}} < 2.
$$

We conclude that the number of primitive medium solutions of \eqref{17.1}  for each $\alpha\in \bm T$ is $\mathcal{O}(N)$. 
In a similar way, the number of primitive medium solutions of \eqref{17.2}  for each $\alpha^\ast\in \bm T^\ast$ is $\mathcal{O}(N)$. 
By Proposition \ref{thesetT}, the number of primitive medium solutions of \eqref{Tin} is $\mathcal{O}(N)$.  We obtain Proposition \ref{pmedium},  as by our definition we have  \eqref{bN} for  $n<s^2$, and  $N=2$ when $n\geq s^2$.


\section{Part II: Strategy, outline and definitions}\label{secpart2def}

We will consider binary forms $F(x, y) \in C(4s-2)$ (see \eqref{C(t)}, for definition). By Lemma \ref{SchL2}, our discussion for such forms implies Theorem \ref{mainsmall}.

\textbf{Definition of Normalized and  Reduced Forms.}
Suppose $0 \leq \mathfrak{a} \leq \mathfrak{b}$. The number of primitive solutions to the inequality
$$
\mathfrak{a} \leq \left| F(x , y)\right| \leq \mathfrak{b}
$$
remains unchanged  if we replace $F$ by one of its $\textrm{GL}_{2}(\mathbb{Z})$-equivalent forms.  Moreover, 
if the inequality $\mathfrak{a} \leq \left| F(x , y)\right| \leq \mathfrak{b}$ has at least one primitive solution $(x_{1}, y_{1})$, there is an 
$A \in \textrm{GL}_2(\mathbb{Z})$ with $A^{-1} (x_{1}, y_{1})^{\textrm{tr}} = (1 , 0)$, so that 
$$
\mathfrak{a} \leq \left| F_{A}(1 , 0)\right| \leq \mathfrak{b}.
$$
So in order to estimate the number of primitive solutions to the above inequality, we may restrict our attention to \emph{normalized forms} for which the leading coefficient $a_{0}$ has 
$$
\mathfrak{a} \leq \left| a_{0}\right| \leq \mathfrak{b}.
$$
We will say that a form $F$ is \emph{reduced} if it is normalized and has the smallest Mahler measure among all normalized forms equivalent to $F$.

Schmidt worked with reduced forms to establish the results in \cite{S87}, in particular, his Lemma 4 is of special importance here (see our Lemma \ref{SchL4} and its implications). 

Every primitive solution of the inequality $1 \leq |F(x , y)| \leq h$ is either a  solution of
\begin{equation}\label{in1h}
h^{1/2}< \left| F(x , y)\right| \leq h
\end{equation}
or
\begin{equation}\label{in2h}
1\leq\left| F(x , y)\right| \leq h^{1/2}.
\end{equation}
To obtain our desired bounds, we will need to assume that the form $F$ in \eqref{in1h} is normalized  with respect to $(\mathfrak{a},\mathfrak{b})=(h^{1/2},h)$ and the form $F$ in \eqref{in2h}  is normalized with respect to 
$(\mathfrak{a},\mathfrak{b})=(1,h^{1/2})$.  We will have two equivalent, but not identical forms in each of the inequalities. We will show that the number of primitive solutions to each of these two inequalities is $\mathcal{O}(ns)$, provided that $h$ satisfies  
\begin{equation}\label{hinD}
 h < \dfrac{|D| ^{\frac{1}{4(n-1)} }}{10^n \, n^{\frac{n}{4(n-1)}}}.
 \end{equation}

From now on, we consider the inequality
\begin{equation}\label{ineqab}
\mathfrak{a} \leq \left| F(x , y)\right| \leq \mathfrak{b}
\end{equation}
 and assume that $F$ is reduced with respect to $(\mathfrak{a},\mathfrak{b})$, 
 \begin{equation}\label{blessh}
 \mathfrak{b} \leq  h,
 \end{equation}
 and 
 \begin{equation}\label{bovera}
\frac{ \mathfrak{b}} {\mathfrak{a}} \leq h^{1/2}.
\end{equation}
These assumptions are necessary in our estimation of linear  forms in Section \ref{secpart2small}. Namely, the assumption \eqref{blessh} is important for our estimates in the proof of Lemma \ref{SchL5new}, and 
and the assumption \eqref{ineqab} is essential for Lemma \ref{SchL4} to hold.

We denote by $M$ the Mahler measure $M(F)$ of the reduced form $F$.

  By \eqref{mahD5}, our assumption \eqref{hinD} implies
 \begin{equation}\label{implied-h}
 h < \dfrac{M^{1/2}}{10^n},
 \end{equation}
and
\begin{equation}\label{SchE3.4}
 M>100^n h^2.
 \end{equation}

Similar to Schmidt's work in \cite{S87}, we define
\begin{equation}\label{SchE5.1}
 Q := \frac{M}{h} \geq 100^n h.
\end{equation}

If $(x , y)$ is a solution to our inequality, we will assume without loss of generality that $y > 0$ (recall that we count $(x', y')$ and $(-x', -y')$ as one solution and we 
have at most one primitive solution with $y=0$). We take
 $$Y_{S}' = M^2.
 $$

\textbf{Definition of small and large solutions.} We call a  solution $(x , y) \in \mathbb{Z}^2$  small if $0 < y \leq Y_S'$. We call a  solution $(x , y) \in \mathbb{Z}^2$  large if  $Y_S' < y$.

We will prove the following in the remaining of the manuscript.
\begin{prop}\label{p2small}
Let $\mathcal{N}_1(F, n, s, h)$ be the number of  small solutions of $1\leq \left|F(x , y)\right| \leq h$, where $F(x, y) \in C(4s -2)$ of degree $n$ is reduced. Assume that  $h$ satisfies  \eqref{hinD}, with \eqref{ineqab} and \eqref{blessh}. We have 
$\mathcal{N}_1(F, n, s, h) \ll \sqrt{ns}$.
\end{prop}

For $n\gg 1$, $h$ satisfying \eqref{hinD}, and our quantity $Y_S' = M^2$, we have 
$$
Y_S' > (2h(2n^\frac{1}{2}M)^n)^{\frac{1}{n-\lambda}} (4(Me^{\frac{n}{2}})^{\frac{1}{a^2}})^{\frac{\lambda}{n-\lambda}},
$$
where the right hand side is the quantity defined by Schmidt in \cite[eq. 4.6]{S87} to distinguish between  small and large solutions,
where $a$ and $\lambda$ are defined as in \eqref{deflambda}.  Therefore we may apply Schmidt's  upper bound for the number of large solutions to our inequalities $\mathfrak{a} \leq \left| F(x , y)\right| \leq \mathfrak{b}$. 
We note here that in \cite{S87}, no restriction on $h$ is assumed, however our assumption \eqref{hinD} results in having the above inequality for $Y_S'$ hold.

 \begin{prop}\label{p2large}
Let $\mathcal{N}_2(F, n, s, h)$ be the
 number of  large solutions of $1\leq \left|F(x , y)\right| \leq h$, where $F(x, y) \in C(4s -2)$ has degree $n$. Assume that $h$ satisfies  \eqref{hinD}. We have
$\mathcal{N}_2(F, n, s, h)\ll \sqrt{ns}$.
\end{prop}
\begin{proof}
This is proven in \cite{S87}. See the discussion in the beginning of page 247 of \cite{S87}, which results in  Theorems 3 and 4 of \cite{S87}. 
\end{proof}

We conclude this section by recalling the trivial but important fact that  in inequalities \eqref{in1h} and \eqref{in2h} the discriminant $D$ is fixed, but the Mahler measure (and therefore the definition of small and large solutions) varies.

\section{Small solutions (II), proof of Proposition \ref{p2small}}\label{secpart2small}

We first give an upper bound for the number of solutions such that $0< y\leq Y_S'$. 

\subsection{Estimation of linear forms} 
Let $0 \leq \mathfrak{a} \leq \mathfrak{b}$, with \eqref{ineqab} and \eqref{blessh}. Suppose that $F(x , y)$ belongs to the class $C(4s-2)$, is reduced and satisfies \eqref{SchE5.1}.

We have
$$
F(x,y)=a_0(x-\alpha_1y)\ldots(x-\alpha_n y),
$$
where $\alpha_1,\ldots, \alpha_n$ are the roots of the polynomial $F(x,1)$, and 
put 
$$
L_i(x,y)=x-\alpha_i y
$$
for $i=1,\dots,n$.  The following is Lemma 4 of \cite{S87}. We present its short proof here to clarify the definition of $Q$ in \eqref{SchE5.1} and more importantly the importance of the assumption $\mathfrak{b} \leq h$.
\begin{lemma}\label{SchL4}
Suppose $G(x , y) =  b_{0} (x -\beta_{1} y)\ldots (x - \beta_{n} y)$ is normalized and equivalent to the reduced form $F$, with \eqref{ineqab}, \eqref{blessh} and \eqref{bovera}, and let
\begin{equation}\label{SchE5.3}
\eta_{i} = |\beta_{i} - m| + 1\, \, \, (i = 1, \ldots, n)
\end{equation}
where $m$ is an integer. Then
\begin{equation}\label{SchE5.4}
\eta_{1} \ldots \eta_{n}  > Q,
\end{equation}
where $Q$ is given in  in \eqref{SchE5.1}.
\end{lemma}
\begin{proof}
The form 
$$
\hat{G}(x , y) = G(x + my , y) = b_{0} \prod_{i=1}^{n} (x + (m - \beta_{i}) y)
$$
is also normalized and equivalent to both $G$ and $F$. Since $F$ is reduced 
$$
M(F) \leq M(\hat{G}) < \left|b_{0}\right| \eta_{1} \ldots \eta_{n}.
$$
Our proof is complete, since $G$ is reduced and  $\left|b_{0} \right| \leq h$. 
\end{proof}

Our next Lemma is  a modified version of Lemma 5 of \cite{S87}.
\begin{lemma}\label{SchL5new}
Suppose  $(x_{0}, y_{0})$ and $(x , y)$ are linearly independent primitive integer points that satisfy 
$\mathfrak{a} \leq \left| F(x , y) \right| \leq \mathfrak{b}$,  with \eqref{ineqab}, \eqref{blessh} and \eqref{bovera}. Then there are numbers $\psi_{1}$, \ldots, $\psi_{n}$ satisfying
\begin{equation}\label{SchE5.5new}
\psi_{i} = 0 \, \, \qquad \textrm{or} \, \, \qquad \frac{1}{2n} \leq \psi_{i} \leq 1,
\end{equation}
with
\begin{equation}\label{SchE5.6new}
\sum_{i=1}^{n} \psi_{i} \geq \frac{1}{2},
\end{equation}
such that 
\begin{equation}\label{SchE5.7new}
\left|\frac{L_{i}(x_{0}, y_{0})}{L_{i}(x , y)} \right|  \geq \left(Q^{\psi_{i}} -\frac{3}{2}  - h^{1/2n} \right)  \left|x_{0} y - x y_{0}\right|,
\end{equation}
for $i \in \{1, \ldots, n\}$.
\end{lemma}
\begin{proof}
Pick $(x', y') \in \mathbb{Z}^2$ with $x'y -x y' = 1$, so that $(x', y')$ and $(x , y)$ form a basis for $\mathbb{Z}^2$. We may write
$
(x_{0}, y_{0}) = a (x , y)+ b (x', y')$. Then
$$
x_{0}y - xy_{0} = b\left(x' y -xy'\right)=b.
$$
Therefore,
\begin{equation}\label{SchE5.8new}
\frac{L_{i}(x_{0}, y_{0})}{L_{i}(x , y)} = a + (x_{0} y - x y_{0}) \frac{L_{i}(x', y')}{L_{i}(x , y)}  = a - (x_{0} y - x y_{0}) \beta_{i}, 
\end{equation}
for $i \in \{1, \ldots, n\}$. (We define $\beta_i$ by the second equation above.)
Set
$$
G(v , w) : = F\left(v(x, y) + w(x', y')\right),
$$
so that $G$ is equivalent to $F$, and $G$ is normalized (recall that $(x, y)$ and $(x', y')$ are fixed and $x'y -x y' = 1$). We have
\begin{eqnarray*}
G(v , w)  &= & a_{0} \prod_{i=1}^{n} L_{i}\left(v(x, y) + w(x', y')\right) \\
& = & a_{0} \prod_{i=1}^{n} \left( v L_{i}(x, y) + w L_{i} (x', y')\right) \\
& = & b_{0}\prod_{i=1}^{n} \left( v + \frac{L_{i}(x', y')}{L_{i}(x , y)} w  \right)\\
& = & b_{0} \prod_{i=1}^{n} (v - \beta_{i} w)
\end{eqnarray*}
with $b_{0} = F(x , y)$. Note that 
$$
 \left|\frac{F(x_{0} , y_{0})}{F(x , y)}\right|\leq h^{1/2}.
$$
We may assume that 
$$
\left| \frac{L_{n}(x_{0}, y_{0})}{L_{n}(x, y)}\right| = \min_{i} \left| \frac{L_{i}(x_{0}, y_{0})}{L_{i}(x, y)}\right|,
$$
so that 
$$
\left| \frac{L_{n}(x_{0}, y_{0})}{L_{n}(x, y)}\right| \leq h^{\frac{1}{2n}}.
$$
By \eqref{SchE5.8new}, we have
\begin{equation}\label{SchE5.9new}
\left| a - (x_{0} y - x y_{0}) \beta_{n}\right| \leq h^{\frac{1}{2n}},
\end{equation}
and
$$
\left| a - (x_{0} y - x y_{0}) \beta \right| \leq h^{\frac{1}{2n}},
$$
where $\beta$ is the real part of $\beta_{n}$. Now let $m$ be an integer with 
$$
| m - \beta| \leq \frac{1}{2}
$$
and define $\eta_{1}, \ldots, \eta_{n}$ by \eqref{SchE5.3}, so that \eqref{SchE5.4} holds by Lemma \ref{SchL4}. We define
\[
\eta'_{i} = \left\{ \begin{array}{lcl}
Q  & \mbox{if}  & \eta_{i} \geq Q,\\
\eta_{i}  & \textrm{if}  & Q^{1/2n} \leq\eta_{i} <Q,\\
1   & \textrm{if}  & \eta_{i} < Q^{1/2n}.
\end{array} \right.
\]
We note that since  $\eta_{1} \ldots \eta_{n} > Q$, we have, by the definition, that
\begin{equation}\label{eta'Q}
\eta'_{1} \ldots \eta'_{n} \geq Q^{1/2}.
\end{equation}
Now we define the numbers $\psi_{i}$, for $i \in \{1, \ldots, n\}$, as follows
$$
\eta'_{i} = Q^{\psi_{i}}.
$$
Clearly $\psi_{i}$ satisfy \eqref{SchE5.5new} and \eqref{SchE5.6new}. By \eqref{SchE5.8new} and \eqref{SchE5.9new},
\begin{eqnarray*}
\left|\frac{L_{i}(x_{0}, y_{0})}{L_{i}(x , y)} \right| & = & \left| (\beta - \beta_{i}) (x_{0} y - xy_{0}) + a - (x_{0} y - xy_{0}) \beta \right|\\
&\geq& \left| \beta - \beta_{i}\right| \left| x_{0} y - xy_{0}\right| -h^{\frac{1}{2n}} \geq \left( \left| m - \beta_{i}\right| - \frac{1}{2} \right) 
\left| x_{0} y - xy_{0} \right| - h^{\frac{1}{2n}}   \\
&\geq &     \left( \left| m - \beta_{i}\right| - \frac{1}{2}  - h^{\frac{1}{2n}} \right) \left| x_{0} y - xy_{0} \right| \\
& = & \left( \eta_{i} -\frac{3}{2}  - h^{\frac{1}{2n}} \right)\left| x_{0} y - xy_{0} \right|.
\end{eqnarray*}
Since $\eta_i \geq \eta'_i = Q^{\psi_{i}}$, the proof of the lemma is completed.
\end{proof}

The following is  a modified version of Lemma 6 in \cite{S87}:

\begin{lemma}\label{SchL6new}
Suppose $(x , y)$ is primitive, with $y> 0$, $\mathfrak{a}\leq |F(x , y)| \leq \mathfrak{b}$, with \eqref{ineqab}, \eqref{blessh} and \eqref{bovera}.
Then there are numbers $\psi_{i} = \psi_{i}(x , y)$ ($i = 1, \ldots, n$), which satisfy \eqref{SchE5.5new} and \eqref{SchE5.6new}, such that
\begin{equation}\label{SchE5.10new}
\left| L_{i}(x , y)\right| < \frac{1}{Q^{\psi_{i}/2} y}
\end{equation}
for each $i$ with $\psi_{i} > 0$.
\end{lemma}
\begin{proof}
We first note, by definition,  that for $i$ with $\psi_{i} > 0$, we have $\psi_{i} \geq \frac{1}{2n}$, and by \eqref{SchE5.1},
$$
Q^{\psi_{i}} \geq Q^{\frac{1}{2n}} > 3 +2h^{1/2n}.
$$
Therefore,
$$
Q^{\psi_{i}} - \frac{3}{2} - h^{1/2n} \geq \frac{1}{2} Q^{\psi_{i}}\geq Q^{\psi_i/2}
$$
since, by \eqref{SchE5.1} again, $Q\geq 4^{2n}$.
We now apply Lemma \ref{SchL5new} with $(x_{0}, y_{0}) = (1 , 0)$.
\end{proof}

\subsection{Counting  Small Solutions}

We define $\Phi_i$ ($i=1,\ldots,n$) by
\begin{equation}\label{phi}
 \left\{\begin{array}{ll}
                \Phi_i=0 &\mbox{when $|\textrm{Im} \, \alpha_{i}|>1$},\\
                M^{-\Phi_i}=|\textrm{Im} \, \alpha_{i}| &\mbox{when $0< |\textrm{Im} \, \alpha_{i}|\leq 1$},\\
                \Phi_i=+\infty &\mbox{when $\alpha_i$ is real}.
               \end{array}\right.
\end{equation}

\begin{lemma}\label{SchL7}
Let $\mathfrak{Y}$ be the set of primitive integer points satisfying  
$$\mathfrak{a} \leq |F(x , y)| \leq \mathfrak{b},
$$
with $\frac{\mathfrak{b}}{\mathfrak{a}}\leq h^{\frac{1}{2}}$, and $0 < y \leq Y_{S}'$. Then for $i = 1, \ldots, n$,
\begin{equation}\label{SchE6.1}
\sum_{x \in \mathfrak{Y}} \psi_{i}(x , y) < 10 \min (1, \Phi_{i}) \frac{\log M}{\log Q}.
\end{equation}
\end{lemma}
\begin{proof} 
We follow the proof of Lemma 7 in \cite{S87}.
For a fixed $i$, let $(x_{1}, y_{1})$, \ldots, $(x_{\nu}, y_{\nu})$ be the elements of $\mathfrak{Y}$ with $\psi_{i}> 0$, 
ordered such that $y_{1}\leq \ldots \leq y_{\nu}$. By \eqref{SchE5.10new}, we have
$$
\left| \textrm{Im} \, \alpha_{i}\right| \leq \left| \alpha_{i} - \frac{x_{j}}{y_{j}}\right| < \frac{1}{Q^{\psi_{i}(x_{j}, y_{j})/2} y_{j}^2}
$$
for $j=1, \ldots, \nu$. First we conclude that $\left|\textrm{Im} \, \alpha_{i}\right| \leq 1$, so that
\begin{equation}\label{phi}
\left|\textrm{Im} \, \alpha_{i}\right| = M^{-\Phi_{i}}
\end{equation}
(with $M^{-\infty} = 0$). So we have
$$
 M^{-\Phi_{i}} < \frac{1}{Q^{\psi_{i}(x_{j}, y_{j})/2} y_{j}^2}.
 $$
Therefore, for every $(x_{j}, y_{j}) \in \mathfrak{Y}$, we have
\begin{equation}\label{SchE6.2}
y_{j} < M^{\Phi_{i}/2}, \qquad \, \, \psi_{i}(x_{j}, y_{j}) < 2 \frac{\log M}{\log Q} \Phi_{i}.
\end{equation}
In particular,
$$
y_{\nu} \leq Y_{i},
$$
where 
\begin{equation}\label{defofYi}
Y_{i} : = \min(Y_{S}' , M^{\Phi_{i}/2}).
\end{equation}
Now let us suppose that $\nu > 1$ and $1 \leq j < \nu$. We have
$$
\left|\alpha_{i}- \dfrac{x_{j}}{y_{j}}\right| < \dfrac{1}{Q^{\psi_{i}(x_{j}, y_{j})/2} y_{j}^2},
$$
and 
$$
\left|\alpha_{i}- \dfrac{x_{j+1}}{y_{j+1}}\right| < \dfrac{1}{Q^{\psi_{i}(x_{j+1}, y_{j+1})/2} y_{j+1}^2}.
$$
So we have
\begin{eqnarray*}
1& \leq & \left| x_{j} y_{j+1} - x_{j+1} y_{j} \right| \\
& = &\left| \left(y_{j} y_{j+1}\right) \left( \frac{x_{j}}{y_{j}} - \alpha_i + \alpha_{i} - \frac{x_{j+1}}{y_{j+1}} \right)\right|\\
& < & \frac{y_{j+1}}{y_{j} Q^{\psi_{i}(x_{j}, y_{j})/2}} + \frac{y_{j}}{y_{j+1} Q^{\psi_{i}(x_{j+1}, y_{j+1})/2}}\\
& < &  \frac{y_{j+1}}{y_{j} Q^{\psi_{i}(x_{j}, y_{j})/2}} + \frac{1}{3},
\end{eqnarray*}
since $\frac{y_{j}}{y_{j+1}}\leq 1$ and $Q^{\psi_{i}(x_{j+1}, y_{j+1})/2} \geq Q^{\frac{1}{4n}} > 3$ by \eqref{SchE5.1}. 
Therefore, we obtain the following \emph{gap principle}.
$$
y_{j+1} > \frac{2}{3} Q^{\psi_{i}(x_{j}, y_{j})/2} y_{j} > Q^{\psi_{i}(x_{j}, y_{j})/4} y_{j} \, \qquad (1\leq j < \nu),
$$
by using \eqref{SchE5.1} once again.

Applying the above gap principle repeatedly, and by the definition of $Y_{i}$ in \eqref{defofYi}, we have
$$
Q^{\frac{1}{4}\left(\psi_i(x_{1}, y_{1}) + \ldots + \psi_i(x_{\nu-1}, y_{\nu-1})\right)} < y_{\nu} \leq Y_{i}
$$
and consequently,
\begin{equation*}
\sum_{j=1}^{\nu -1} \psi_{i}(x_{j}, y_{j}) < 4 \frac{\log Y_{i}}{\log Q}.
\end{equation*}
For our choice of $Y_{S}' = M^{2}$,
$$
\log Y_{S}' = 2 \log M,
$$
and therefore
$$
\sum_{j=1}^{\nu -1} \psi_{i}(x_{j}, y_{j}) < 4 \min\left(2, \frac{\Phi_{i}}{2}\right) \frac{\log M}{\log Q} \leq 8 \min\left(1, \Phi_{i}\right) 
\frac{\log M}{\log Q}.
$$
Now to estimate $\sum_{j=1}^{\nu} \psi_{i}(x_{j}, y_{j}) $, we only need to estimate  $\psi_{i}(x_{\nu}, y_{\nu}) $. By \eqref{SchE5.5new} and \eqref{SchE6.2}, we have
$$
\psi_{i}(x_{v}, y_{v}) \leq \min \left(1 , 2 \Phi_{i}\frac{\log M}{\log Q}\right) \leq  2 \min\left(1, \Phi_{i}\right) \frac{\log M}{\log Q}.
$$
So we conclude the assertion of the Lemma.
\end{proof}

Now we note that the number of small solutions to $\mathfrak{a}\leq |F(x,y)|\leq \mathfrak{b}$ is equal to $\sum_{x \in \mathfrak{Y}} 1$, and by 
\eqref{SchE5.6new} and \eqref{SchE5.1}, we have
\begin{eqnarray*}
\sum_{x \in \mathfrak{Y}} 1& \leq& 2 \sum_{i=1}^{n} \sum_{x \in \mathfrak{Y}} \psi_{i}\\
&\leq& 20\sum_{i=1}^{n} \min\left(1, \Phi_{i}\right) \frac{\log M}{\log Q}\\
&\leq& 40\sum_{i=1}^{n} \min\left(1, \Phi_{i}\right).
\end{eqnarray*}
In the next subsection we will show that
\begin{equation}\label{sumphi}
\sum_{i=1}^{n} \min\left(1, \Phi_{i}\right) \ll \sqrt{sn}
\end{equation}
This will complete the proof of  Proposition \ref{p2small}.

\subsection{The clustering of roots with small imaginary parts}
In order to utilize a powerful result of Schmidt in \cite{S87}, which will be stated in Proposition \ref{corsec9}, 
 we will assume that $n>1700 (\log n)^3$. Otherwise,  we have $n \leq 1700 (\log n)^3$ and therefore in this case $n \ll s$, and 
 the previously established  bound $\mathcal{O}(n)$ for the number of primitive solutions of general Thue's inequalities (see \cite{Sh15}, for example)  will prove Proposition \ref{p2small}.

Our goal now is to show \eqref{sumphi}, for $n>1700 (\log n)^3$. Our discussion is the same as Section 9 of \cite{S87}.
\begin{prop}\label{corsec9}
Let $f(z)$ be a polynomial of degree $n$ with rational coefficients, of Mahler height $M(f)$ and without multiple roots. Suppose that $f(x) f'(x)$ has 
not more than $q - 1$ real roots, where $q \geq1$. Suppose further that $M(f) > e^{2n}$. Then for $\phi$ in
\begin{equation}\label{SchE9.5}
1700 n^{-1} (\log n)^3 \leq \phi \leq 1,
\end{equation}
the number of roots $\mathfrak{x} + \mathbf{i} \mathfrak{y}$ with imaginary part in $0 < \mathfrak{y} \leq M(f)^{-\phi}$ does not exceed 
$\left( \frac{8n q}{\phi}\right)^{1/2}$.
\end{prop}
\begin{proof} This is the Corollary in Section 9 of \cite{S87}.
 \end{proof}

Since $F\in C(4s-2)$, then  $f(z)=F(z,1)$ has no more than $4s-1$ real zeros and $f'(z)=F_x(z,1)$ has no more than $4s-2$ real zeros. So $ff'$ has  at most $8s-3=q-1$ real zeros, where we take
\begin{equation}\label{qleq}
q = 8s-2.
\end{equation}

We may suppose that $\Phi_1\geq \ldots\geq\Phi_n$.
The number of summands in
\begin{equation}\label{SchE9.6}
\sum_{i=1}^{n} \min(1 , \Phi_{i})
\end{equation}
with $\Phi_{i} \geq 1$ is the number of roots $\alpha_i$ of $F(x,1)$ with $|\mathrm{Im}\, \alpha_i|\leq M^{-1}$. By  taking $\phi=1$ in \eqref{SchE9.5}, 
the contribution of summands with $\Phi_{i} \geq 1$ in the sum \eqref{SchE9.6} is 
$$
< \left( 8n q \right)^{1/2} = 2 \sqrt{2nq}.
$$

Clearly the summands with $\Phi_i\leq 1700 n^{-1}(\log n)^3<1$ do not contribute.

The remaining summands have $1700 n^{-1}(\log n)^3<\Phi_i<1$. Since
$|\mathrm{Im}\, \alpha_j|=M^{-\Phi_j}\leq M^{-\Phi_i}$ for $j\leq i$, Proposition \ref{corsec9} yields $i<(8nq/\Phi_i)^{1/2}$, so $\Phi_i<8nq/i^2$.
We conclude that these terms contribute
\begin{eqnarray*}
&< &\sum_{i=1}^n \min(1,8nq/i^2) =[2\sqrt{2nq}] +8nq\sum_{i=[\sqrt{8nq}]+1}^n 1/i^2\\
& <& 4\sqrt{2nq} + 1 -8q \ll \sqrt{ns},
\end{eqnarray*}
by \eqref{qleq}.


\section*{Acknowledgements}
The authors are very grateful to the anonymous referee for very helpful suggestions. This article was written while Akhtari was a visitor at the \emph{Max Planck Institute for Mathematics in Bonn}. Akhtari acknowledges the MPIM constant support for her research and collaboration, and in particular both authors are grateful for the opportunity to meet and advance this project in Bonn.
The authors acknowledge the 
support from  FIM, \emph{Forschungsinstitut f\"ur Mathematik}, and are thankful to FIM for  the hospitality and providing stimulating research atmosphere, especially during Akhtari's visits to Z\"urich in 2018 and 2019.
Akhtari's research is partially  supported by the NSF  grant DMS-1601837 and \emph{Simons Foundation's collaboration grant for mathematicians}.
Bengoechea's research is supported by SNF grant 173976.


\end{document}